\documentclass[twoside,11pt,reqno]{amsart}
\usepackage{amscd, bbm}
\usepackage{amssymb,mathabx}
\usepackage{mathrsfs,wasysym, tikz}
\usetikzlibrary{matrix, arrows}
\usepackage{url}

\usepackage[pdftex,colorlinks,hypertexnames=false,linkcolor=blue,urlcolor=blue,citecolor=blue]{hyperref}

\usepackage{amsmath}

\makeatletter

\@addtoreset{equation}{section}

\numberwithin{equation}{section}

\newtheorem{Theorem}{Theorem}[section]
\newtheorem*{Theorem*}{Theorem}
\newtheorem{Lemma}[Theorem]{Lemma}
\newtheorem{Proposition}[Theorem]{Proposition}
\newtheorem{Corollary}[Theorem]{Corollary}

\theoremstyle{definition}
\newtheorem{Definition}[Theorem]{Definition}

\theoremstyle{remark}
\newtheorem{Remark}[Theorem]{Remark}
\newtheorem*{proofofDM}{Proof of Theorem~\ref{PDME}}
\newtheorem{Question}[Theorem]{Question}
\newtheorem{Example}[Theorem]{Example}

\newbox\squ  
\setbox\squ=\hbox{\vrule width.6pt
   \vbox{\hrule height.6pt width.4em\kern1ex
         \hrule height.6pt}%
   \vrule width.6pt}

\renewcommand{\k}{\mathbbm{k}}

\newcommand{\g}{\mathfrak{g}}

\newcommand{\p}{\mathfrak{p}}

\renewcommand{\P}{{\mathcal{P}}}

\newcommand{\ad}{\operatorname{ad}}

\newcommand{\Ker}{\operatorname{Ker}}

\newcommand{\Spec}{\operatorname{Spec}}

\newcommand{\End}{\operatorname{End}}

\newcommand{\Mat}{\operatorname{Mat}}

\newcommand{\Cas}{\operatorname{Cas}}

\newcommand{\Der}{\operatorname{Der}}

\newcommand{\PSpec}{\operatorname{\mathcal{P}-Spec}}

\newcommand{\Frac}{\operatorname{Frac}}
\newcommand{\si}{{\operatorname{sc}}}

\newcommand{\oa}{\overline{a}}

\newcommand{\ox}{\overline{x}}
\newcommand{\olambda}{\overline{\lambda}}

\title[]{\boldmath On the semi-centre of a Poisson algebra}
\author{Cesar Lecoutre \& Lewis Topley}
\email{C.Lecoutre@kent.ac.uk}
\email{L.Topley@kent.ac.uk}

\begin{document}

\maketitle

\begin{abstract}
If $\g$ is a Lie algebra then the semi-centre of the Poisson algebra $S(\g)$ is the subalgebra generated by $\ad(\g)$-eigenvectors.
In this paper we abstract this definition to the context of integral Poisson algebras.
We identify necessary and sufficient conditions for the Poisson semi-centre $A^\si$ to be a Poisson algebra graded by its weight spaces. 
In that situation we show the Poisson semi-centre exhibits many nice properties: the rational Casimirs are quotients of Poisson normal elements and the Poisson Dixmier--M{\oe}glin equivalence holds for $A^\si$. 
\end{abstract}

\section{Introduction}

Throughout this paper $\k$ is a field of characteristic zero, all vector spaces are defined over $\k$, and $\g$ will be a Lie algebra.
The symmetric algebra $S(\g)$ carries a natural structure of a Poisson algebra. 
It is easy to see that the subalgebra $S(\g)^\g \subseteq S(\g)$ consisting of elements annihilated by $\ad(\g)$ coincides with the Poisson centre. 
The semi-invariants are, by definition, the common eigenvectors for $\ad(\g)$ and the algebra $S(\g)^\si$ which they generate is known as the Poisson semi-centre. 
This is a Poisson commutative subalgebra of $S(\g)$ graded by the weight space decomposition of $\ad(\g)$.

Over the years the study of semi-centres has motivated a sizable body of research, see \cite{Di,DDV,JS,J6,J7,Ooms} and the references therein. Since this topic arose in the context of invariant theory some of the central questions are the polynomiality and factoriality of semi-centres.
One notable outlet for the study of semi-invariants lies in the computation of the rational invariants of $S(\mathfrak{g})$. 
By the results of Rentschler and Vergne, Dixmier's fourth problem is in fact equivalent to the statement that the centre of $\Frac S(\mathfrak{g})$ is purely transcendental over $\k$, see \cite{RV} and \cite[Probl\`emes]{Di}.
Thanks to \cite{DDV} every rational invariant is a quotient of two elements of $S(\mathfrak{g})^\si$ with the same weight, and so the theory of semi-invariants appears naturally in some important classical problems.


The purpose of this article is to define and study of the Poisson semi-centre $A^\si$ of an arbitrary integral Poisson algebra $A$,
by which we mean a Poisson algebra which is also an integral domain.
We recall that a Poisson normal element $a\in A$ is such that $\{A,a\} \subseteq Aa$, equivalently the principal ideal $Aa$ is Poisson, and our first observation is that when $A = S(\g)$
the Poisson semi-invariants of $S(\g)$ are precisely the same as Poisson normal elements (Lemma~\ref{pnorm}).
With this in mind we define \emph{the Poisson semi-centre $A^\si$ of $A$} to be the subalgebra generated by the Poisson normal elements. 
In general this subalgebra need not be a Poisson subalgebra (see Example \ref{notP}), and even when it is, it need not be Poisson graded by the weight spaces for the Hamiltonian derivations (see Example \ref{notW}).
To remedy this we begin the paper by identifying a necessary and sufficient condition for $A^\si$ to be a Poisson algebra graded by the Poisson weight space decomposition, as we now explain.

Since $A$ is assumed to be a domain, it is easily shown that for every Poisson normal element $a \in A$ there exists a Poisson derivation $\lambda : A \rightarrow A$ such that 
\[\{b,a\} = \lambda(b) a \quad\text{for all}\quad b \in A\qquad\text{(Lemma~\ref{PoissStr})}.\] 
The additive submonoid of $\Der_\k(A)$ generated by these derivations will be denoted $\Lambda(A)$. 
In Proposition~\ref{AWP} we show that $A^\si$ is a Poisson subalgebra of $A$ graded by the weight space decomposition if and only if $\Lambda(A)$ is an abelian Lie submonoid of $\Der_\k(A)$, and we refer to the latter condition as the {\em abelian weight property}. 
We exhibit several large families of Poisson algebras satisfying this property, including symmetric algebras $S(\g)$ of Lie algebras, Poisson affine spaces, semiclassical limits of various quantised coordinate rings (see \cite{GL}) and the algebras $A(n,a)$ studied by Sierra and the first author in \cite{LS}.

Motivated by the close connection between the semi-centre $S(\g)^\si$ and the centre of the Poisson quotient field $\Frac S(\g)$ (see \cite{DDV}) we investigate the relationship between Poisson ideals, normal elements and the centre of the fraction field of the semi-centre.
Some of our results are gathered together here; see Propositions~\ref{PoissHom} and \ref{RatCas}.
\begin{Proposition}\label{combilemma}
Let $A$ be an integral Poisson algebra with the abelian weight property and such that $A^\si$ is finitely generated. Then the following hold:
\begin{enumerate}
\item[(i)] Every nonzero Poisson ideal of $A^\si$ contains a nonzero Poisson normal element;
\item[(ii)] Every rational Casimir of $A^\si$ is a quotient of two normal elements weighted by the same derivation.
\end{enumerate} 
\end{Proposition}

If $A$ is an integral Poisson algebra and $a,b\in A$ are normal elements weighted by the same derivation, then it is easily seen that $ab^{-1}$ lies in the centre of $\Frac A$.
\begin{Question}
Does every element of the centre of $\Frac A$ arise as the quotient of two normal elements?
\end{Question}

When $A$ is a given Poisson algebra, the study of the rational Casimirs
is a challenging problem; this is especially true for Dixmier's
fourth problem in the case of symmetric algebras of Lie algebras (see \cite{Ooms} for a detailed discussion).
One application of such information is to develop our understanding of the Poisson primitive ideals of $\PSpec(A)$ via the Poisson
Dixmier--M{\oe}glin equivalence. Recall that a Poisson prime ideal $I \subseteq A$ is called locally closed if $\{I\}$ is a locally
closed subset of the Poisson spectrum $\PSpec(A)$; $I$ is called Poisson primitive if it is the largest Poisson ideal
contained in some maximal ideal of $A$; finally, $I$ is called rational if the Poisson centre of the quotient field of $A/I$ is algebraic over $\k$.
Thanks to \cite[1.7, 1.10]{Oh} we know that every locally closed ideal is primitive and every primitive ideal is rational. Brown and Gordon asked whether
all three properties might coincide \cite{BGo}, and when they do we say that $A$ satisfies the Poisson Dixmier--M{\oe}glin equivalence.
Using Proposition~\ref{combilemma} we prove the following.
\begin{Theorem}\label{PDMEg}
Let $A$ be an integral Poisson algebra with the abelian weight property and such that $A^\si$ is finitely generated. Then the Poisson Dixmier--M{\oe}glin equivalence holds for $A^\si$.
\end{Theorem}

We now describe the structure of this paper. 
In \S \ref{s1} we discuss the definition of the Poisson semi-centre and the abelian weight property, showing that some familiar examples of Poisson algebras satisfy this property.
In \S \ref{s2} we consider a class of finitely generated Poisson algebras axiomatising the algebras $A^\si$ where $A$ is an integral Poisson algebra with the abelian weight property.
We call these Poisson algebras generalised Poisson affine spaces and we prove Proposition~\ref{combilemma} in the context of such algebras, from which we deduce Theorem~\ref{PDMEg}.

\medskip

\noindent {\bf Acknowledgements:} We would like to thank Professor David Jordan and Professor Alfons Ooms for carefully reading the first draft of this manuscript and making helpful suggestions.
The second author is grateful for the support of EPSRC grant EP/N034449/1. 


\section{The Poisson semi-centre and the abelian weight property}
\label{s1}

Suppose that $\g$ is a Lie algebra over $\k$. If $\{x_i \mid i \in I\}$ is a basis for $\g$ then the symmetric algebra $S(\g)$ carries a natural structure of a Poisson algebra with bracket:
\begin{eqnarray}\label{e:poissb}
\{f, g\} = \sum_{i,j \in I} \frac{\partial f}{\partial x_i} \frac{\partial g}{\partial x_j} [x_i, x_j] 
\end{eqnarray}
for $f,g\in S(\g)$. The invariants of $S(\g)$ are the elements $S(\g)^\g := \{f\in S(\g) \mid \ad(\g) f = 0\}$ and the semi-invariants are defined to be $$\{f\in S(\g) \mid \ad(\g) f \subseteq \k f\}.$$ An easy calculation using \eqref{e:poissb} shows that the Poisson centre of $S(\g)$ is equal to $S(\g)^\g$.
The algebra which is generated by the set of all semi-invariants is known as the semi-centre $S(\g)^{\si}$, and it has been the focus of much research over the years.
The Poisson normal elements of $S(\g)$ are defined to be the elements $f\in S(\g)$ such that $\{S(\g), f\} \subseteq S(\g) f$.
\begin{Lemma}\label{pnorm}
The Poisson normal elements of $S(\g)$ are precisely the semi-invariants.
\end{Lemma}

\begin{proof}
If $a$ is a semi-invariant then using \eqref{e:poissb} we see that $a$ is Poisson normal.
Conversely if $a$ is Poisson normal, then for any $x\in\g$ there exists $\lambda(x) \in A$ such that $\{x,a\}=\lambda(x)a$. 
We deduce that $\lambda(x) \in\k$ from the fact that the Poisson bracket \eqref{e:poissb} satisfies $\{S(\g)_i,S(\g)_j\}\subseteq S(\g)_{i+j-1},$ where $S(\g) = \bigoplus_{i \geq 0} S(\g)_i$
is the grading with $\g$ placed in degree 1.
\end{proof}

The above discussion leads us naturally to:
\begin{Definition}
The semi-centre of a Poisson algebra $A$ is the subalgebra $A^\si$ generated by Poisson normal elements.
\end{Definition}

One nice feature of the semi-centre $S(\g)^\si$ is that it is Poisson commutative. 
However outside the Lie theoretic setting, this fails immediately.
To illustrate what may go wrong we present a couple of examples.
The first one shows that in general $A^\si$ is not necessarily a Poisson subalgebra of $A$.

\begin{Example}
\label{notP}
Let $A=\k[x,y,z]$ with brackets $\{x,y\}=xyz$, $\{x,z\}=x$ and $\{y,z\}=y$. Then $A^\si=\k[x,y]$ is not closed under the Poisson bracket.
\end{Example}

The following example shows that even when $A^\si$ is a Poisson subalgebra it is not always Poisson commutative.

\begin{Example}\label{PAeg}
Let $A = \k[x_1,...,x_n]$ be a polynomial algebra and let $(\lambda_{i,j})_{1\leq i,j \leq n} \in \Mat_n(\k)$ be a skew-symmetric matrix.
Define a Poisson bracket on $A$ by the rule
\begin{eqnarray}\label{PoissAff}
\{x_i, x_j\} = \lambda_{i,j} x_i x_j
\end{eqnarray}
This algebra is known as \emph{Poisson affine space} and, since the generators $x_1,...,x_n$ are Poisson normal we have $A = A^\si$ is not Poisson commutative in general.
The \emph{Poisson torus} $T$ associated to $A$ is the localisation of $A$ at the generators
\[T=\k[x_1^{\pm1},...,x_n^{\pm1}]\]
and the Poisson bracket on $A$ extends uniquely to a Poisson bracket on $T$.
\end{Example}

We proceed to discuss the properties of normal elements. 
Recall that a Poisson derivation $\lambda \in \Der_\P(A)$ is a $\k$-derivation of $A$ which is also a derivation of the Lie bracket $\{\cdot, \cdot\}$ of $A$.
\begin{Lemma}\label{PoissStr}
If $A$ is an integral domain and $a\in A$ is Poisson normal then there exists a Poisson derivation $\lambda \in \Der_\P(A)$ such that
\begin{eqnarray}
\{b, a\} = \lambda(b) a.
\end{eqnarray}
\end{Lemma}
\begin{proof}
Since $a$ is normal we have $\{b, a\} = \lambda(b) a$ for some linear map $\lambda : A \rightarrow A.$ We must check that
$\lambda$ is a Poisson derivation. For $b,c \in A$ we have
\begin{eqnarray*}
\lambda(bc)a = \{bc, a\} = \{b, a\} c + b\{c,a\} = (\lambda(b)c + b\lambda(c))a,
\end{eqnarray*}
and
\begin{eqnarray*}
\lambda(\{b,c\})a
&=& \{\{b,c\}, a\} = \{\{b, a\}, c\} + \{b, \{c, a\}\\ &=& \{\lambda(b)a, c\} + \{b,\lambda(c)a\}\\
&=& \{\lambda(b), c\} a + \lambda(b) \{a, c\} + \{b, \lambda(c)\}a + \lambda(c)\{b, a\}\\
&=& (\{\lambda(b), c\} + \{b, \lambda(c)\})a.
\end{eqnarray*}
Since $A$ is integral we conclude that 
\begin{eqnarray*}
\lambda(bc)=\lambda(b)c + b\lambda(c); \text{ and} \\
\lambda(\{b,c\}) = \{\lambda(b), c\} + \{b, \lambda(c)\}
\end{eqnarray*}
as required.
\end{proof}

From henceforth we assume that $A$ is an integral Poisson algebra.
For any $\lambda \in \Der_\P(A)$ we make the notation
$$A_\lambda := \{a \in A \mid \{b, a\} = \lambda(b) a \text{ for all } b\in A\}$$
and write $$\Lambda(A) := \{\lambda \in \Der_\P(A) \mid A_\lambda \neq 0\}.$$
Since $\{A,\k\} = 0$ we have $0 \in \Lambda(A)$ and when $a \in A_\lambda$ and $b\in A_\mu$
we have $ab \in A_{\lambda + \mu}$ by the Jacobi identity, so that $\Lambda(A)$ is a commutative
submonoid of $\Der_\P(A)$. This leads to an alternative description of the semi-centre
\begin{eqnarray}\label{thegrading}
A^\si = \bigoplus_{\lambda \in \Lambda(A)} A_\lambda.
\end{eqnarray}
The derivations $\lambda \in \Lambda(A)$ will be referred to as \emph{the weights of $A$}, whilst the
subspaces $A_\lambda$ will be called \emph{the weight spaces}.

Although the formula \eqref{thegrading} defines a grading on $A^\si$ as an associative subalgebra of $A$, it does not in general define a Poisson grading (see Example \ref{notW}).
In this paper we are interested in the case where $A^\si$ is a Poisson subalgebra
which is Poisson graded by \eqref{thegrading}, ie. $\{A_\lambda, A_\mu\} \subseteq A_{\lambda + \mu}$ for $\lambda, \mu \in \Lambda$.
The following translates these properties into statements about $\Lambda$.
\begin{Proposition}\label{AWP}
Let $A$ be an integral Poisson algebra. Then the following are equivalent:
\begin{enumerate}
\item[(i)] $[\Lambda, \Lambda] = 0$;
\item[(ii)] $A^\si$ is a Poisson subalgebra of $A$ and \eqref{thegrading} is a Poisson grading.
\end{enumerate}
If {\rm(i)} or {\rm(ii)} holds then we say that $A$ has \emph{the abelian weight property}. Furthermore $\lambda(A_{\mu})\subseteq A_{\mu}$ for all $\lambda,\mu\in\Lambda$.
\end{Proposition}
\begin{proof}
Let $\lambda, \mu \in \Lambda$, pick $x \in A_\lambda, y\in A_\mu$ and $a\in A$. We have
\begin{eqnarray*}
\{a, \{x, y\}\} &=& \{\{a, x\}, y\} + \{x, \{a, y\}\} = \{\lambda(a) x, y\} + \{x, \mu(a) y\} \\
&=& \{\lambda(a), y\} x + \lambda(a) \{x, y\} + \{x, \mu(a)\} y + \mu(a)\{x,y\}\\ &=&  [\mu, \lambda](a) xy + (\mu + \lambda)(a)\{x, y\}.
\end{eqnarray*}
If $[\Lambda, \Lambda] = 0$ then $\{A_\lambda,A_\mu\}\subseteq A_{\lambda+\mu}$ and $A^\si$ is a Poisson algebra graded by its weight spaces \eqref{thegrading}.
The converse is clear from the equality
\[0=\{a, \{x, y\}\}-(\mu + \lambda)(a)\{x, y\}=[\mu, \lambda](a) xy\]
since $A$ is integral.

Let $a\in A_\lambda$ be nonzero, and let $b\in A_\mu, c\in A$. Then
\begin{eqnarray*}
\{c, \lambda(b)\}a &=& \{c, \lambda(b)a\}-\{c,a\}\lambda(b) \\
& = & \{c, \{b,a\}\}-\lambda(c)\lambda(b)a\\
& = & (\lambda+\mu)(c)\{b,a\}-\lambda(c)\lambda(b)a\\
& = & (\lambda+\mu)(c)\lambda(b)a-\lambda(c)\lambda(b)a= \mu(c)\lambda(b)a.
\end{eqnarray*}
Since $A$ is an integral domain it follows that $\lambda(b) \in A_\mu$ and so $\lambda$ preserves the weight spaces for all
$\lambda \in \Lambda$.
\end{proof}
\begin{Remark}
When $A$ has the abelian weight property, the Poisson normal elements of $A^\si$ are precisely the elements homogeneous
with respect to the grading $A^\si = \bigoplus_\lambda A_\lambda$. We shall use these two names interchangeably for such
elements. We also point out that homogeneous elements of degree zero are the same as Poisson central elements.
\end{Remark}

The abelian weight property is reasonably natural as the next result illustrates.
\begin{Proposition}\label{exampless}
Let $A$ be a Poisson algebra and $S$ a multiplicative set in $A$. If $AS^{-1}$ has the abelian weight property then so does $A$. 
Furthermore the following families of Poisson algebras satisfy the abelian weight property:
\begin{enumerate}
\item[(i)] the symmetric algebras of Lie algebras;
\item[(ii)] Poisson affine spaces and Poisson tori, described in Example~\ref{PAeg};
\item[(iii)] semiclassical limits of  the following quantum algebras defined in \cite[Sections 2.3-2.7]{GL}:
\begin{itemize}
\item quantum matrices;
\item quantum symplectic spaces;
\item quantum euclidean spaces;
\item quantised Weyl algebras;
\item quantum (anti-)symmetric matrices;
\end{itemize}
\item [(iv)] the Poisson algebras $A(n,a)$ from \cite{LS}.
\end{enumerate}
\end{Proposition}
\begin{proof}
There is a Lie algebra embedding form $\Der_\P(A)$ into $\Der_\P(AS^{-1})$ extending derivations via the Leibniz rule. 
If $a$ is normal in $A$ with weight $\lambda$ then the following computation shows it is also normal in $AS^{-1}$,
and that the weight is the image of $\lambda$ in $\Der_\P(AS^{-1})$
\[\{bs^{-1},a\}=-bs^{-2}\{s,a\}+s^{-1}\{b,a\}=\big(\lambda(b)s^{-1}-bs^{-2}\lambda(s)\big)a.\]
Thus $\Lambda(A) \hookrightarrow \Lambda(AS^{-1})$ as abelian groups, which proves the first claim.

We now verify that the examples listed in the proposition satisfy the abelian weight property:

(i) Let $\g$ be a Lie algebra. Since the maps $\lambda \in \Lambda(S(\g))$ are derivations and $S(\g)$ is generated by $\g$
it suffices to show that $[\lambda, \mu](\g) = \{0\}$ for all $\lambda, \mu \in \Lambda(S(\g))$. By Lemma~\ref{pnorm} the Poisson normal elements of
$S(\g)$ are actually semi-invariants and so $\lambda, \mu$ send $\g \rightarrow \k$. 
It follows that $\lambda \circ \mu (\g) = \mu \circ \lambda(\g) = \{0\}$ and as a result $[\lambda, \mu](\g) = \{0\}$.

(ii) Now let $A = \k[x_1,...,x_n]$. For $i = 1,...,n$ we let $x_i \in A_{\lambda_i}$ for $\lambda_1,...,\lambda_n \in \Lambda(A)$ and write
$\partial_i := \frac{\partial}{\partial x_i}$. It follows from (\ref{PoissAff}) that $\lambda_j = \sum_{i=1}^n \lambda_{i,j} x_i \partial_i$
for $\lambda_{i,j} \in \k$.
Since the derivations $\{x_i \partial_i \mid i=1,...,n\}$ pairwise commute it follows immediately that the same is true for $\lambda_1,...,\lambda_n$.
The monoid of weights of the torus $\k[x_1^{\pm1},...,x_n^{\pm1}]$ is generated by the weights $\{\pm\lambda_i \mid i=1,...,n\}$ and so is abelian.

(iii) These Poisson algebras are all Poisson iterated Ore extensions to which the Poisson deleting derivations algorithm \cite{LL} can be applied, and therefore they localise to Poisson tori \cite[Theorem 5.3.2]{L}. The result then follows from the first claim along with part (ii).

(iv) Fix $n\geq1$ and $a\in\k$. By \cite[Lemma 3.26]{LS} the Poisson algebra $A=A(n,a)$ has a localisation $A^\circ$ which is isomorphic to the Poisson algebra $\k[Y_{0}^{\pm1},X,Y_2,...,Y_n]$ with nonzero Poisson brackets
\[\{X,Y_i\}=(a+i)Y_0Y_i.\]
It is straightforward to see that its semi-centre is $\k[Y_{0}^{\pm1},Y_2,\dots,Y_n]$ and that the monoid of weights is generated by the commuting set $\{\pm aY_0^{\pm1}\partial_{X},(a+i)Y_0\partial_{X} \mid i=2,...,n\}$.
\end{proof}

Despite holding for the families described in the proposition, the next example shows that the abelian weight property does not hold for every integral Poisson algebra.

\begin{Example}
\label{notW}
Every Poisson bracket on $A = \k[x,y]$ is determined by a choice of $\{x,y\}$ thanks to the derivation rule and skew-symmetry. 
Furthermore, every possible choice actually defines a Poisson bracket. 
If we define $\{x,y\} = pxy$ for some $p \in A$ then both $x$ and $y$ are normal and so the resulting Poisson structure satisfies $A^\si =A$. 
The weights of $x$ and $y$ are respectively $\lambda_x=-py\partial_{y}$ and  $\lambda_y=px\partial_{x}$.
Since $[\lambda_x,\lambda_y](x)=-\lambda_x(p)x$ and $[\lambda_x,\lambda_y](y)=\lambda_y(p)y$ it follows that $A$ has the abelian weight property if and only if $p\in\k$.
\end{Example}

\section{Generalised Poisson affine space and the Poisson Dixmier--M{\oe}glin Equivelence}
\label{s2}

In this section we investigate algebraic and geometric properties of $\PSpec A^\si$ and so we restrict ourselves to the case where the semi-centre is finitely generated. 
We remark that this is not always the case, as shown in \cite[Section 5]{DDV}.
Our results focus on the case where the semi-centre is a Poisson algebra graded by its weight space decomposition.
In view of Proposition \ref{AWP} the most appropriate way to discuss such algebras seems to be via the following axiomatisation.
\begin{Definition}
We say that a Poisson algebra $A$ is a \emph{generalised Poisson affine space} if
\begin{itemize}
\item[(i)]{$A$ is an integral domain generated over $\k$ by Poisson normal elements $x_1,...,x_n$ with weights $\lambda_1,...,\lambda_n$;}
\item[(ii)]{the weights pairwise commute.}
\end{itemize} 
\end{Definition}
The next lemma shows that the axioms of a generalised Poisson affine space are passed down to all prime Poisson quotients, making them amenable to inductive arguments.
\begin{Lemma}\label{thequotient}
When $A$ is a generalised Poisson affine space and $I$ is a prime Poisson ideal, $A/I$ is a generalised Poisson affine space.
\end{Lemma}
\begin{proof}
Let $x_1,...,x_n$ be Poisson normal generators of $A$. We can suppose that there is $1\leq m \leq n$ such that $\{x_1,...,x_m\} \cap I = \emptyset$ and $\{x_{m+1},...,x_n\} \subseteq I$.
Then the images $\ox_1,...,\ox_m$ in $A/I$ are normal generators in $A/I$. Since the latter is an integral domain Lemma~\ref{PoissStr}
tells us that there are derivations $\olambda_1,...,\olambda_m$ of $A/I$ such that $\{\oa, \ox_i\} = \olambda_i(\oa) \ox_i$ for all
$i=1,...,m$. For all $a\in I$ we have $\{a, x_i\} = \lambda_i(a)x_i \in I$ and since $I$ is prime and $x_i \notin I$ we deduce that $\lambda_i(a) \in I$.
In other words, $\lambda_i(I) \subseteq I$ and the map $\olambda_i$ is just the map induced by $\lambda_i$ on the quotient $A/I$.
Finally, since $\{\lambda_i \mid i=1,...,m\}$ pairwise commute we may conclude that the same is true for $\{\olambda_i \mid i=1,...,m\}$.
\end{proof}
\begin{Example}
The Poisson affine space of Example~\ref{PAeg} is a generalised Poisson affine space, thanks to Lemma~\ref{exampless}. Further examples can be obtained by:
\begin{enumerate}
\item[(i)] forming a Poisson affine space over some ground ring $K$, which is a finitely generated commutative $\k$-algebra;
\item[(ii)] taking a prime Poisson quotient of any generalised Poisson affine space.
\end{enumerate}
In fact when all of the normal generators $x_1,...,x_n$ of a generalised affine space $A$ are prime, it is easy to
show that the Poisson brackets all have the form $\{x_i, x_j\} = \lambda_{i,j} x_i x_j$ where $\lambda_{i,j} \in \Cas(A)$ is a Casimir
for $1\leq i,j \leq n$. In this case the fibres of the map $\Spec A \rightarrow \Spec \Cas(A)$ are Poisson affine spaces.
It is interesting to wonder whether this conclusion holds when the normal generators are not necessarily prime.
\end{Example}

In order to prove Theorem~\ref{PDMEg} we actually prove the following result, which is equivalent.
\begin{Theorem}\label{PDME}
When $A$ is a generalised Poisson affine space the Poisson Dixmier--M{\oe}glin equivalence holds for $A$.
\end{Theorem}

\emph{For the rest of the section we assume $A$ is a generalised Poisson affine space.}
We let $x_1,...,x_n$ be the normal generators with weights $\lambda_1,...,\lambda_n$
and we let $\Lambda$ be the monoid consisting of the weights of normal elements.
Recall that, by Lemma~\ref{AWP} the decomposition $A = \bigoplus_{\lambda \in \Lambda} A_\lambda$ is a Poisson grading.
\begin{Lemma}\label{derivationlemma}
Let $\Lambda_0 \subseteq \Lambda$ be any collection of derivations and suppose that for all $i=1,...,n$ we have $\hash \{\lambda(x_i) \mid \lambda \in \Lambda_0\} = 1$.
Then $\hash \Lambda_0 = 1$.
\end{Lemma}
\begin{proof}
Suppose that $\lambda, \mu \in \Lambda_0$ and observe that $\lambda - \mu \in \Der_\P(A)$. If $\lambda - \mu$ vanishes on the generators $x_1,...,x_n$
then by the Leibniz rule it vanishes on all of $A$. The lemma follows.
\end{proof}

The next proof follows the same principle as Artin's linear independence of characters of a group.
\begin{Proposition}\label{PoissHom}
Every nonzero Poisson ideal contains a nonzero homogeneous element.
\end{Proposition}
\begin{proof}
Let $I$ be a Poisson ideal. For each $a\in I$ we may decompose $a = \sum_{\lambda \in \Lambda} a_\lambda$  with $a_\lambda \in A_\lambda$ and write
$\ell(a) = \hash \Lambda(a)$ where $\Lambda(a) := \{\lambda\in \Lambda \mid a_\lambda \neq 0\}$. We show that $I$ contains an element with $\ell(a) = 1$.
Pick $a \in I$ such that $\ell(a) > 1$ is minimal.
Recall that each $x_i$ is homogeneous of weight $\lambda_i$.
By Proposition~\ref{AWP} the derivations $\lambda \in \Lambda$ preserve the grading and so for any $i\in\{1,...,n\}$ we have 
\[\{x_i, a\}=\sum_{\lambda \in \Lambda(a)} \lambda(x_i) a_\lambda \in \bigoplus_{\lambda\in\Lambda(a)} A_{\lambda+\lambda_i}.\]
From Lemma \ref{derivationlemma} there is some $i \in \{1,...,n\}$ such that $\mu_1(x_i) \neq \mu_2(x_i)$ for some $\mu_1,\mu_2 \in \Lambda(a)$.
Thus, for this choice of $i$, the expression $\mu_1(x_i) a - \{x_i, a\}$ is non-zero, lies in $I$ and has $\ell(\mu_1(x_i) a - \{x_i, a\}) < \ell(a)$.
This contradicts the minimality of $\ell(a)$ and the contradiction proves the claim.
\end{proof}
\begin{Remark}
In general it is not true that every Poisson ideal in a generalised Poisson affine space is generated by homogeneous elements.
For example, let $A := \k[x_1, x_2]$ with Poisson bracket given by $\{x_1, x_2\} = x_1x_2$. It is not hard to see that the Poisson normal elements
are precisely the monomials in $x_1, x_2$, however for all $s \in \k$ the ideal $(x_1, x_2 - s)$ is Poisson.
\end{Remark}

We now recall a few facts about modules over Poisson algebras, required in the proof of Proposition \ref{RatCas}. 
A \emph{Poisson $A$-module} is a vector space $V$ equipped with two linear maps $A \rightarrow \End_\k(V)$, which we write
\begin{eqnarray*}
& & a \longmapsto m(a);\\
& & a \longmapsto \delta(a),
\end{eqnarray*}
such that $m$ is a representation of $A$ as an associative algebra, $\delta$ is a representation of $A$ as a Lie algebra, and
\begin{eqnarray}
& & \delta(ab) = m(a) \delta(b) + m(b) \delta(a);\\
& & m(\{a,b\}) = [\delta(a), m(b)].
\end{eqnarray}
Obviously $A$ is a Poisson module over itself, with $\delta(a)b := \{a,b\}$ and $m(a)b := ab$.
Poisson ideals of $A$ provide more examples of Poisson modules, and yet another source of examples is provided by the fraction field $\Frac(A)$
which inherits a Poisson structure from $A$ and admits $A$ as a Poisson subalgebra.
\begin{Lemma}
Let $W$ be a Poisson $A$-module and let $U, V \subseteq W$ be Poisson submodules.
The set $$(U: V) := \{a\in A \mid m(a) U \subseteq V \text{ and } \delta(a)U \subseteq V\}$$
is a Poisson ideal of $A$.
\end{Lemma}
\begin{proof}
Let $a\in (U:V)$, $b \in A$ and $u\in U$. We have
\begin{eqnarray*}
& & m(ab)u = m(a) (m(b) u) \in V;\\
& & m(\{a,b\})u = [\delta(a), m(b)] u \in V;\\
& & \delta(ab)u = m(a) \delta(b) u + m(b) \delta(a) u \in V;\\
& & \delta(\{a,b\})u = [\delta(a), \delta(b)] u \in V, 
\end{eqnarray*}
and so $ab, \{a,b\} \in (U:V)$.
\end{proof}

The following result says that every rational Casimir is a quotient of two normal elements. Our approach
was inspired by the corresponding statement in symmetric algebras of Lie algebras, first proven in \cite{DDV}.
\begin{Proposition}\label{RatCas}
Consider the set $$Q_\lambda = \{ab^{-1} \mid a,b\in A_\lambda\} \subseteq \Frac(A).\smallskip$$
We have $$\Cas(\Frac(A)) = \bigcup_{\lambda \in \Lambda} Q_\lambda.$$
In other words, every Casimir of $\Frac(A)$ is a quotient of homogeneous elements of $A$ of same weight.
\end{Proposition}
\begin{proof}
The fact that the elements of $Q_\lambda$ are Casimirs follows from a short calculation in $\Frac(A)$,
which we leave to the reader. Let $ab^{-1} \in \Cas(\Frac(A))$ and consider the Poisson $A$-submodule $U \subseteq \Frac(A)$ generated by $ab^{-1}$.
Since $ab^{-1}$ is a Casimir the map $A \rightarrow U$ sending $c$ to $cab^{-1}$ is an isomorphism of Poisson modules.
According to the previous lemma the space $(U : A)$ is a Poisson ideal of $A$. We claim that $(U : A) \neq 0$. For all $c\in A$ we have
\begin{eqnarray*}
& & b^2cab^{-1} = bca \in A;\\
& & \{b^2, cab^{-1}\} = 2b\{b, cab^{-1}\} = 2b\{b,c\}ab^{-1} = 2\{b,c\}a \in A.
\end{eqnarray*}
It follows that $b^2 \in (U:A) \neq 0$. Now we may apply Proposition~\ref{PoissHom} to deduce that $(U: A)$
contains a nonzero homogeneous element $c \in A_\lambda$. By definition we have $cab^{-1} = d \in A$ and since $ab^{-1}$
is a Casimir it follows that $d \in A_\lambda$. Now we have equality $ab^{-1} = dc^{-1}$ in $\Frac(A)$ which shows that $ab^{-1} \in Q_\lambda$.
\end{proof}

\begin{Corollary}
\label{depn}
If $\Cas \Frac A$ is a finite extension of $\k$ then for every $\lambda \in \Lambda$, every two elements $a,b \in A_\lambda$
are algebraically dependent over $\k$, ie. there is a non-zero $f \in \k[X,Y]$ such that $f(a,b) = 0$.
\end{Corollary}
\begin{proof}
Suppose that $a,b\in A_\lambda$ are algebraically independent. We claim that the set
$$\left\{\frac{a}{b - sa} \mid s \in \k\right\}$$
is a $\k$-linearly independent subset of $\Cas \Frac(A)$. Since $\k$ is a field of characteristic zero
it has infinite cardinality and so this claim will prove the lemma. 
Since these are fractions of Poisson normal elements of the same weight $\lambda$
they are Casimirs as claimed. Suppose that $s_1,...,s_n \in \k$
are distinct elements and suppose that $t_1,...,t_n \in \k$ are some elements such that
$$\sum_i t_i \frac{a}{b-s_ia} = 0.$$
Clearing the denominators and using the fact that $\k[a,b]$ is an integral domain we get
$$\sum_i t_i \prod_{j\neq i} (b - s_j a) = 0.$$
Since this equation holds in the polynomial ring $\k[a,b]$ it holds modulo the ideal $(b- s_j a) \unlhd \k[a,b]$ with $j=1,...,n$.
This gives $t_1 = t_2 = \cdots = t_n = 0$ and this proves the claim.
\end{proof}

\begin{proofofDM}
By \cite[1.7, 1.10]{Oh} the locally closed ideals are Poisson primitive and the latter ideals are rational, so it suffices to check that rational ideals are locally closed.
If $I$ is a Poisson prime ideal of $A$, then $A/I$ is a generalised Poisson affine space by Lemma~\ref{thequotient}. 
After replacing $A$ by $A/I$, this reduces the proof to showing that if $\Cas \Frac A$ is a finite extension of $\k$ then the zero ideal in $A$ is locally closed.

Let $I$ be an nonzero Poisson prime ideal. We claim that $I$ contains at least one of the generators $x_1,...,x_n$. By Proposition~\ref{PoissHom}
we know that there is a nonzero element $a\in I \cap A_\lambda$ for some $\lambda \in \Lambda$. By Lemma~\ref{AWP}
the monoid $\Lambda$ is finitely generated by the weights $\lambda_1,...,\lambda_n$ of the generators $x_1,...,x_n$
and so we may assume that $\lambda = \sum_{i=1}^n m_i \lambda_i$ for non-negative integers $m_1,...,m_n$.
It follows that $b := x_1^{m_1}\cdots x_n^{m_n} \in A_\lambda$.

Consider the polynomial ring $R := k[X, Y]$ and define a $\Lambda$-grading on $R$ by placing both $X$ and $Y$ in degree $\lambda$.
Consider the homomorphism
\begin{eqnarray*}
\phi & : & R \longrightarrow A;\\
& & X \longmapsto a;\\
& & Y \longmapsto b.
\end{eqnarray*}
It is evidently a homogeneous morphism with respect to the $\Lambda$-gradings on $A$ and $R$, and so the kernel is homogeneously
generated. Furthermore, by Corollary~\ref{depn}, $\Ker \phi \neq 0$, and so we can choose $f(X, Y) = \sum_{i=0}^d s_i X^i Y^{d-i}$
where $s_0, ...,s_d \in \k$ such that $f(a,b) = 0$. If $s_i = 0$ for $i < d$ then this relation says that $s_d a^d = 0$. Since $A$ is an integral domain
we know that this is not the case. Hence $s_i \neq 0$ for some $i < d$. Suppose that $m= \min\{ i \mid s_i \neq 0\}$ and observe that

$$f(a,b) = a^{m} \sum_{i=m}^d s_i a^{i-m} b^{d-i} = 0.$$

\noindent Using the fact that $A$ is an integral domain once again we see that
$$\sum_{i=m}^d s_i a^{i-m} b^{d-i} = 0$$
with $s_m \neq 0$. We can rewrite this as
$$s_m b^{d-m} = -\sum_{i=m+1}^d s_i a^{i-m} b^{d-i} \in I$$
We have now shown that $I$ contains a monomial of the form
$b=x_1^{m_1} \cdots x_n^{m_n}$. Since $I$ is prime it must be that it contains one of the elements $x_1,...,x_n$ as claimed.

The deductions made above imply that the zero ideal is equal to the following open subset of $\PSpec(A)$:
$$\{(0)\} = \bigcap_{i=1}^n \{P \in \PSpec(A) \mid (x_i) \nsubset P\}.$$
As a consequence $(0)$ is locally closed and the proof is complete. $\hfill \qed$
\end{proofofDM}

\begin{Remark}
In this paper we assumed throughout that $A$ is an integral domain, however this hypothesis can be removed when $A$ is noetherian, reduced
and the minimal prime ideals $\p_1,...,\p_n$ are pairwise coprime. When $A$ is such a Poisson algebra the $\p_1,...,\p_n$ are all Poisson \cite[Lemma~1.1]{Po} and so the natural map
$A \rightarrow A/\p_1 \times \cdots \times A/\p_n$ is a Poisson homomorphism. The map is surjective by the Chinese remainder theorem
and the kernel is $\bigcap_i \p_i = 0$. Now our results can be applied to each of the direct factors $\{A/\p_i \mid i=1,...,n\}$.  Geometrically this
just corresponds to a Poisson variety with disjoint irreducible components.
\end{Remark}

\end{document}